\documentclass{article}

\usepackage[cp1251]{inputenc}            % Выбор языка и кодировки
\usepackage{amsmath}
\usepackage{amssymb}
\usepackage{amsthm}

\usepackage[english, russian]{babel}     % Настройки для русского языка

\usepackage{color}
\definecolor{darkgreen}{rgb}{0,.5,0}
\usepackage[colorlinks,
            filecolor=blue,
            linkcolor=blue,
            citecolor=darkgreen]{hyperref}

\numberwithin{equation}{section}
\theoremstyle{plain}
\newtheorem*{theorem*}{Теорема}
\newtheorem{theorem}{Теорема}
\newtheorem{lemma}{Лемма}[section]
\newtheorem{corollary}{Следствие}
\newtheorem{corollary*}{Следствие}
\theoremstyle{definition}
\newtheorem*{definition*}{Определение}
\newtheorem*{remark}{Замечание}

\begin{document}

УДК 512.552.4 + 512.572

\textbf{Относительно свободные ассоциативные алгебры }

\textbf{ранга 2 и 3 с тождеством Ли нильпотентности и}

\textbf{системы порождающих некоторых $\rm T$-пространств}

\textbf{С.В. Пчелинцев}

\textbf{Аннотация.} Изучаются относительно свободные ассоциативные алгебры $F_{r}^{(n)}$ ранга $r=2,3$ с тождеством $[x_1,...,x_n]=0$ Ли нильпотентности степени $n\ge 3$ над полем $K$ характеристики $\ne 2,3$. Доказана теорема о произведении $T^{(m)} T^{(n)}$ для ассоциативной алгебры $A$ ранга 3: $T^{(m)}(A)\cdot T^{(n)}(A)\subseteq T^{(m+n-1)}(A)$, где $T^{(n)}$~--- $\rm T$-идеал в $A$, порожденный коммутатором $[x_1,...,x_n]$; ограничение на ранг существенно.
Получено описание тождеств от трех переменных алгебры $F^{(n)} $, из которого, в частности, следует, что $Z(F_{r}^{(n)}) = (T^{ (n-1)}+Z_{q}) (F_{r}^{(n)})$, где  $p={\rm char} (K)\ge 5$, $q-$наименьшая степень $p$ такая, что $q\ge n-1$ и $Z_{q}$~--- $\rm  T$-пространство, порожденное $x^{q} $. Доказано также равенство $Z(F_{2}^{ (n)} )=F_{2}^{(n)} \cap Z(F^{(n)})$.
Доказано, что в алгебре $F_{2}^{(n)}, n\ge 4$ над полем $K$ характеристики $p\geq n$ существует конечный убывающий «композиционный» ряд $\rm  T$-идеалов $T^{(3)} = T_1 \supset T_{2} \supset ...T_k \supset T_{k+1} =0$ такой, что все факторы $T_i /T_{i+1} $ не содержат собственных $\rm T$-пространств.

Библиография: 16 наименований.

\textit{Ключевые слова:} тождество Ли нильпотентности, центр, ядро, собственный многочлен, тождество от 2-х переменных, $\rm T$-пространство.

\textbf{Введение}

Всюду в работе рассматриваются только \textit{ассоциативные} алгебры с единицей 1 над бесконечным полем $K$ характеристики, отличной от 2 и 3. Введем следующие обозначения:

$F$~--- свободная ассоциативная алгебра над множеством $X= \{x_1,x_2 ,\dots\}$ свободных порождающих; $F_{r}$~--- подалгебра в $F$, порожденная множеством $X_{r} = \{x_1,\dots,x_{r} \}$;

$ [x_1,\dots,x_n ]-$правонормированный коммутатор степени $n\ge 2$, т.е. $ [x_1,x_2 ]=x_1x_2 -x_2 x_1$
и по индукции $ [x_1,\dots,x_n]= [[x_1,\dots ,x_{n-1} ],x_n]$;

${\rm  LN} (n):[x_1,\dots ,x_n]=0$---тождество Ли нильпотентности степени $n$;

$T^{(n)} = ( [x_1,\dots,x_n])^{T} $ и $V^{ (n)} = \{[x_1,\dots,x_n ]\}^{T}$~--- $\rm T$-идеал и $\rm T$-пространство соответственно, порожденные коммутатором степени $n$;

$F^{(n)} = F/T^{(n)}$~--- относительно свободная алгебра с тождеством ${\rm  LN} (n)$; $F_2^{ (n)} $ и $F_{3}^{(n)}$~--- ее подалгебры ранга 2 и 3 в $F^{(n)} $, порожденные множествами $X_2$ и $X_{3}$ соответственно (мы считаем также, что $X_2 = \{x,y\},X_{3} = \{x,y,z\}$);

$Z^{*}  (A)-$ядро $A$ (наибольший идеал $A$, содержащийся в её центре $Z(A)$).

Напомним, что многочлен $f\in F$ является \textit{собственным}, если он содержится в подалгебре, порожденной коммутаторами вида
$[x_1, \dots ,x_m]$, где $x_1,\dots,x_m \in X$ и $m\ge 2$.

В работах [1] - [13] с разных точек зрения изучались алгебры $F^{(n)} $ при $n\ge 3$. В [1] и [2] были построены аддитивные базисы в алгебрах $F^{(n)} $ для $n=3,4$; а в [8] и [10] были найдены их центры. Задача построения аддитивного базиса в алгебре
$F_2^{(n)}$ отмечалась в [9].

В [12] сначала были описаны собственные центральные многочлены от 2-х переменных алгебр $F^{ (5)} $ и $F^{ (6)} $.  Затем в [13] были описаны собственные центральные многочлены алгебр $F^{ (5)} $ и $F^{ (6)} $ над полем характеристики 0; оказалось, в частности, что центр алгебры $F^{ (5)} $ содержится в T-идеале, порожденном коммутатором степени 4. Задача описания центров этих алгебр в полном объеме пока остается открытой.

В данной работе изучаются алгебры $F_2^{ (n)} $ и $F_{3}^{ (n)} $ ранга 2 и 3 с тождеством ${\rm  LN} (n)$ Ли нильпотентности произвольной степени $n$. Работа состоит из пяти параграфов.

В \S 1 доказана теорема 1 о произведении $T^{(m)}T^{(n)}$ для ассоциативной алгебры $A$ от трех порождающих (ранга 3):
$T^{(m)}(A)\cdot T^{(n)}(A)\subseteq T^{(m+n-1)}(A)$. Ограничение на ранг является существенным, поскольку $[x,y][z,t]\notin T^{(3)}$. Отметим также, что при доказательстве теоремы 1 используется теорема о произведении $T^{(m)}T^{(n)}$ из [12].

В \S 2 получено описание тождеств от трех переменных алгебры $F^{(n)} $ (см. теорему 2); более точно, указан аддитивный базис алгебры $F_{3}^{(n)}$; он получается факторизацией базиса Пуанкаре-Бирхгофа-Витта свободной ассоциативной алгебры $F_{3}$ по идеалу $T^{(n)}.$

В \S 3 содержится описание центров $Z(F_r^{(n)} )$ алгебр $F_r^{(n)}$ ранга $r=2$ и $r=3$. Так, в случае поля $K$ характеристики 0 доказано, что для указанных значений ранга $r$ верно равенство $Z(F_r^{(n)}) = T^{(n-1)}(F_r^{(n)})$, в частности, всякий центральный элемент в алгебре $F_r^{(n)}$ является ядерным, т.е. порождает центральный идеал. Заметим, что если $r\ge 4$, то такого свойства нет даже в алгебре $F_r^{(3)}$. Кроме того, получено описание центра алгебры $F_{3}^{ (n)} $ над полем конечной характеристики $p\ge 5$ (см. теорему 4). Наконец, показано, что $Z(F_2^{(n)}) = F_2^{(n)} \cap Z(F^{(n)})$. Заметим, что для алгебр $F_{3}^{(n)}$ аналогичное утверждение неверно.

В теории многообразий ассоциативных алгебр центральное место занимает проблема Шпехта о конечной базируемости тождеств любой ассоциативной алгебры над полем характеристики 0. Положительное решение проблемы Шпехта получено А.Р. Кемером [14]. В [15] им же положительно решена, так называемая, локальная проблема Шпехта над бесконечным полем конечной характеристики.

Хорошо известны также результаты о конечной и бесконечной базируемости $\rm T$-пространств (см. [3] - [6]). Над бесконечным полем характеристики 2 примеры бесконечно базируемых $\rm T$-пространств были анонсированы А.В. Гришиным [7]. Над бесконечными полями конечной характеристики В.В. Щиголев доказал существование бесконечно базируемых $\rm T$-пространств [4] по модулю $\rm T$-идеала $T^{(3)}$. Существование бесконечно базируемых $\rm T$-идеалов доказано в работах [5] и [6].

Задача изучения $\rm T$-пространств, содержащихся в этаже $T^{(n)} /T^{(n+1)}$ для свободной алгебры $F_2$ ранга 2 была поставлена А.В. Гришиным. Он же получил описание $\rm T$-пространства $T^{(3)} /T^{(4)}$ для алгебры $F_2$ [11].

\S 4 и \S 5 посвящены описанию порождающих систем $\rm T$-пространств, содержащихся в 2-порожденной алгебре $F_2^{ (n)}$ с тождеством Ли нильпотетности степени $n$. Так, в теореме 5 описаны независимые системы порождающих $\rm T$-пространств, содержащихся в коммутанте алгебры $F_2^{(3)}$ (обобщение теоремы В.В. Щиголева [4]). Теорема 6 утверждает, что в алгебре $F_2^{(n)}  (n\ge 4)$ на полем $K$ характеристики $p\geq n$ существует конечный убывающий «композиционный» ряд $\rm T$-идеалов $T^{(3)} = T_1 \supseteq T_2 \supseteq \dots T_k \supseteq T_{k+1} =0$ такой, что все факторы $T_i /T_{i+1} $ не содержат собственных $\rm T$-пространств.

\section{Теорема о произведении $T^{ (m)} T^{(n)} $ для алгебр ранга 3}

В 1965 г. В.Н. Латышев [1] доказал, что для любых натуральных\textit{ $m,n\ge 2$ }справедливо включение $T^{(m)} \cdot T^{(n)} \subseteq T^{(m+n-2)} $ в алгебре $F$ (над полем характеристики $\neq 2$).

\begin{lemma} \label{l:Lat}
В алгебре $F$ для любого $x\in F$ верно:

а) $ [V^{(m)} ,x] [V^{ (n-1-m)} ,x]\subseteq T^{ (n)} $;

б) $T^{(m)} T^{(n)} \subseteq T^{(m+n-2)}$.
\end{lemma}

\begin{lemma} \label{l:GrPc1}
Если хотя бы одно из чисел $m,n\geq2$  нечетно, то в свободной алгебре $F$ выполнено включение
$$ T^{(m)} \cdot T^{(n)} \subseteq T^{(m+n-1)};$$
там же показано, что указанное ограничение на числа $m,n$ является существенным.
\end{lemma}

В 1978 г. И.Б. Воличенко [2] доказал лемму \ref{l:GrPc1} при условии, что одно из чисел $m,n$ равно $3$. В работе [12] это утверждение называлось \emph{теоремой о произведении} $T^{(m)}T^{(n)}$.

Важный элемент доказательства теоремы о произведении заключался в следующем утверждении (см. [12], лемма 2).
\begin{lemma}\label{l:GrPc2}
Для любых $a,b\in F$ справедливо включение
$$[T^{(n)},a,b] + [T^{(n)}, [a,b]]\subseteq T^{(n+2)}.$$
\end{lemma}

Напомним, что в любой ассоциативной алгебре справедливы тождества:
\begin{gather}
  \label{eq:id1}
  [xy,z]= [x, yz] + [y, zx],\\
  \label{eq:id2}
  [xy,z]=x[y,z]+[x,z]y.
\end{gather}

\begin{theorem} \label{th:prod}
  Во всякой 3-порожденной алгебре $A$ верно включение:
$$T^{(m)} (A) \cdot T^{(n)} (A)\subseteq T^{(m+n-1)} (A),$$
где $T^{(m)}$ обозначает $T^{(m)}(A)$.
\end{theorem}
\begin{proof}
Без ограничения общности, можно считать, что $A$ является свободной алгеброй ранга 3. Итак, надо понять, что для любых $m\ge n\ge 2$  верно включение
$$T^{(m)} \cdot T^{(n)} \subseteq T^{(m+n-1)}.$$

Сначала рассмотрим случай: $m$ четно и $n=2$. Проведем индукцию по $m$. Основание индукции при $m = 2$ верно, поскольку $[x,z][x,y] \in T^{(3)}$ и по модулю $T^{(3)} $ коммутант $T^{(2)} $ как идеал порождается множеством $[X_{3} ,X_{3} ].$ Если $m=4$, то в силу центральности элемента Холла в алгебре $F^{(5)} $ (см. [12]), имеем: $[[x,y]^{2},z,t] \in T^{(5)}$. Следовательно,

\begin{equation}\label{eq:id3}
 [x,y,z,t][x,y] \in T^{(5)}.
\end{equation}

Положим $V^{(2)}=[A,A]$ и определим по индукции $V^{(k+1)} = [V^{(k)},A]$.
Поскольку в силу тождества (\ref{eq:id1}) по модулю $V^{(3)}$ верно сравнение
$$[ab,c] \equiv [a,bc]+[b,ac],$$
то $V^{(4)} V^{(2)}$ по модулю $T^{(5)}$ представимо в виде линейной комбинации элементов $w:=[x,z_1 ,z_2 ,z_3 ][y,z]$, где $x,y,z\in X_{3}$. Если $z_3 \in \{y,z\}$, то $w\in T^{(5)}$, значит, можно считать, что $z_3 =x$. Тогда в силу (\ref{eq:id3}) по модулю $T^{(5)}$ имеем:

\[w \equiv[x,t_1 ,t_2 ,x][y,z] \equiv [x,t_1 ,t_2 ,y][z,x]\equiv - [x,z,t_2 ,y][t_1 ,x].\]

По лемме \ref{l:GrPc1} можно считать, что $t_1 \in X_3 $. Тогда по-доказанному $w \equiv 0$.

Заметим, что в алгебре $A$ верно тождество

\begin{equation}\label{eq:id4}
[xy,a,b]=[x,a,b]y+x[y,a,b]+[x,b][y,a]+[x,a][y,b].
\end{equation}

Допустим теперь, что $v_{s} \in V^{(s)} ,a,b,c,d\in A$. Тогда по модулю $T^{(m+1)} $ в силу тождество (\ref{eq:id4}),
лемм \ref{l:GrPc1} и \ref{l:GrPc2} и индуктивного предположения имеем
\[[[v_{m-3},a],b,a][c,d] = [[v_{m-3} ,a][c,d],b,a]-[v_{m-3} ,a][c,d,b,a]-\]
\[-[v_{m-3},a,a][c,d,b]-[v_{m-3} ,a,b][c,d,a] \equiv 0.\]
Итак, доказано, что элемент $w:=[v_{m-3} ,a,t,b][c,d]$ кососимметричен по переменным $a,b,c,d\in A$ по модулю $T^{(m+1)} $. Кроме того, по модулю $T^{(m+1)} $ элемент $w$ является дифференцированием по переменной $d$, значит, можно считать, что $a,b,c,d\in X_{3} $, но тогда $w \in T^{(m+1)} $. Таким образом, доказано, что

\begin{equation}\label{eq:id5}
V^{(m)} V^{(2)} \subseteq T^{(m+1)}.
\end{equation}

Используя включение (\ref{eq:id5}), легко завершить доказательство теоремы. В самом деле, допустим, что $V^{(m)} V^{(n)} \subseteq T^{(n+m-1)} $ для данного четного $n$ и произвольного четного $m$. Тогда в силу (\ref{eq:id5}) и леммы \ref{l:GrPc2} имеем:
$$V^{(m)} [V^{(n)},x,y] \subseteq [V^{(m)} V^{(n)},x,y] - V^{(m+2)} V^{(n)} + T^{(m+n+1)}$$
$$\subseteq T^{(m+n+1)}.$$
\end{proof}

\begin{remark}
  Поскольку $[x,y] [z,t]\notin T^{(3)}$, то ограничение в теореме \ref{th:prod} на ранг алгебры является существенным.
\end{remark}

\begin{corollary}
Верно включение $T^{(n-1)} (F_{3}^{(n)})\subseteq Z(F_{3}^{(n)})$.
\end{corollary}

\begin{proof}
Пусть $T^{(m)} = T^{(m)} (F_{3}^{(n)}),V^{(m)} = V^{ (m)} (F_{3}^{ (n)} )$. Тогда по теореме \ref{th:prod}:
\[ [T^{ (n-1)} ,F_{3}^{(n)} ]= [V^{ (n-1)} F_{3}^{(n)},F_{3}^{ (n)} ]\subseteq T^{ (n)} +V^{ (n-1)} T^{(2)} =0.\]
\end{proof}

\section{Тождества алгебры $F_{3}^{(n)}$}
\subsection{Понятие веса многочлена.}

Пусть $X_3=\{x_1 ,x_2 ,x_3 \}$; $F_3$ и $L_3$~--- свободная ассоциативная и свободная алгебра Ли над множеством $X_3$ соответственно. Пусть $e_1 ,\dots,e_{\alpha } ,\dots$~--- базис алгебры Ли $L_3$, состоящий из однородных элементов. При этом предполагается, что базис упорядочен по возрастанию индексов и $e_{\alpha } <e_{\beta }$, если $\deg (e_{\alpha })>\deg (e_{\beta})$.
Из теоремы Пуанкаре-Бирхгофа-Витта (см. [16]) вытекает, что всякий элемент алгебры $F_3$ однозначно представим в виде линейной комбинации \textit{правильных слов} вида
 \begin{equation}\label{e1}
 v_i = e_{i_1 } e_{i_2 } \dots e_{i_t }, \text{где}\quad t\ge 1,\; e_{i_1} \leq \dots \leq e_{i_t};
\end{equation}
в слове $v_i$ элементы $e_{i_{j}}$ перемножаются в алгебре $F_3$. Иначе говоря, правильные слова составляют аддитивный базис алгебры $F_3$ над полем $K$.

Определим \textit{вес} $ {\rm wt}(e_i)$ (the weight) однородного элемента $e_i\in L_3$ как его степень:
$$ {\rm wt}(e_i)= {\rm deg} _{L_3} (e_i).$$
Если $v$ правильное слово указанного вида (\ref{e1}), то положим
$$ {\rm wt}(v_i)=\sum _j {\rm wt}(e_{i_j}) - t +1,$$
и назовем ${\rm wt}(v_i)$ \textit{весом} слова $v_i$. \textit{Весом} ${\rm wt}(a)$ элемента $a\in F_3$ назовем наименьший из весов правильных слов $v_i$, входящих в разложение

$$a= \sum _i \alpha _i v _i, \text{где}\quad 0\ne \alpha _i\in K.$$

Заметим, что если $f\in T^{(n)}$, то ${\rm wt}(f) \geq n$.

\subsection{Теорема о тождествах от 3-х переменных.}
Нам потребуется следующая лемма
\begin{lemma}\label{l:oper}
 Всякий элемент идеала $T^{(m)}$ представим в виде линейной комбинации элементов вида
 \begin{equation} \label{e2}
xM_1 M_2 \ldots M_k,
 \end{equation}
где $M_i$ является либо оператором правого умножения $R_y$, либо оператором внутреннего дифференцирования $D_y$  и среди операторов $M_i$ содержится не менее $(m-1)$-го оператора вида $D_y$, причем $x,y \in X$.
\end{lemma}

 \begin{proof}
 Из определения идеала $T^{(m)}$ следует, что $T^{(m)}$ линейно порождается элементами вида
\begin{equation} \label{e3}
a_1D_{a_2}\dots D_{a_m}R_{a_{m+1}},
 \end{equation}
где $a_i\in F$. Из тождеств (\ref{eq:id1}) и (\ref{eq:id2}) вытекает, что в ассоциативной алгебре справедливы равенства
\begin{equation} \label{e4}
R_{ab}=R_aR_b,\quad D_{ab}=R_aD_b +L_bD_a,\quad L_a = R_a - D_a.
 \end{equation}
Используя (\ref{e4}), по индукции легко понять, что элемент вида (\ref{e3}) представим в виде линейной комбинации элементов вида (\ref{e2}).
 \end{proof}

\begin{theorem} \label{th:ident}
  В алгебре $F^{(n)}$ выполнено тождество $f(x,y,z)=0$ тогда и только тогда, когда ${\rm wt}(f)\geq n$.
\end{theorem}

\begin{proof}
Доказательство. Покажем сначала, что если ${\rm wt}(f)\geq n$, то $f = 0$ в алгебре $F^{(n)}$. Элемент $f$ представим в виде линейной комбинации правильных слов $p_i$ вида (T9), причем ${\rm wt}(p_i)\geq n$ для любого $i$. Пусть $n_j= {\rm deg}(e_{i_j})$. Тогда по определению веса многочлена имеем
$$n = \left(\sum _{1 \leq j \leq n_i}n_i \right) - t + 1.$$
Далее, алгебре $F$ силу теоремы 1 о произведении:
$$T^{(n_1)} T^{(n_2)}\ldots T^{(n_t )} \subseteq T^{(n)},$$
где
$$n=n_1 +\sum _{2 \leq j \leq t}(n_i -1) = \left(\sum _{1\leq j \leq n_i}n_i \right) - t + 1.$$
Значит, если ${\rm wt}(f)\geq n,$ то $f\in T^{(n)}$ и $f=0$ в алгебре $F^{(n)}$. В частности, если вес правильного слова $p$ вида (T9)
не менее $n$, то $p=0$ в $F^{(n)}$.

Допустим, что $f(x,y,z)=0$ в алгебре $F^{(n)}$, т.е. $f \in T^{(n)}$.
Тогда по лемме \ref{l:oper} элемент $f$ линейно выражается через элементы вида $x':=xM_1 M_2 \dots M_k$,
причем в операторном слове $M_1 M_2 \dots M_k $ встречается не менее, чем $(n-1)$ операторов вида $D_y$.
Пусть $M_k=D_y$. Тогда $x'':=x'M_1 M_2 \dots M_{k-1}$ и в операторном слове $M_1 M_2 \dots M_{k-1}$  встречается не менее, чем $(n-2)$ операторов вида $D_y$. По предположению индукции можно считать, что $x''$ является линейной комбинацией правильных слов
$$p_i= e_{i_1 }e_{i_2} \dots e_{i_t},$$
веса которых $\geq n-1$. Заметим, что
$$p_iD_y = \sum_{j=1}^t e_{i_1 } \dots[e_{i_j},y]\dots e_{i_t}.$$
Из доказательства теоремы Пуанкаре-Бирхгофа-Витта следует, что каждый из элементов $e_{i_1} \dots[e_{i_j},y]\dots e_{i_t}$ является линейной комбинацией правильных слов, веса которых $\geq n$.
\end{proof}

\begin{corollary}
 Если $n\ge 3$, то $[x,y]^{n-1} \in T^{(n)}$ и $[x,y]^{n-2} \notin T^{(n)}$, т.е. индекс нильпотентности коммутатора от свободных порождающих алгебры $F^{(n)}$ равен $n-1$.
\end{corollary}

В самом деле, поскольку ${\rm wt}([x,y]^{m} )= m + 1$, то указанное утверждение немедленно вытекает из теоремы 2. Легко понять, что индекс нильпотентности коммутанта $(F^{(n)})^{'}$ алгебры $F^{(n)}$ также равен $n-1$.

\begin{remark}
Соотношение $ [x,y]^{n-1} \in T^{(n)}$ было ранее доказано в [9, предложение 1]. Вопрос о точности этой оценки в работе [9] не обсуждался.
\end{remark}

\section{Центры алгебр $F_2^{(n)}$ и $F_{3}^{(n)}$}
\subsection{Предварительные леммы.}

Напомним несколько хорошо известных результатов, связанных с биномиальными коэффициентами. Как обычно, $(a,b)$ обозначает наибольший общий делитель чисел $a$ и $b$.

\begin{lemma}\label{l:3.1}
Пусть $p$~--- простое число, $t\geq 1$. Биномиальные коэффициенты $\binom {p^t}{i}$ кратны $p$, если $1\leq i\leq p^t-1$.
  Если же $(p,m)=1$, то $\binom {p^tm}{p^t}$ не делится на $p$.
\end{lemma}

Эта лемма представляет собой простое упражнение.

\begin{lemma}\label{l:3.2}
Пусть $n\geq 2$~--- целое число; $y_1=[x,y]$ и по индукции $y_{k+1}=[y_k,x]$. Тогда справедлив коммутаторный бином Ньютона: $$[x^n,y]=\sum_{i=1}^n \binom {n}{i} x^{n-i}y_i,$$
где $\binom {n}{i}$ --биномиальные коэффициенты.
\end{lemma}

Эта лемма легко получается индукцией по $n$.

\begin{lemma}\label{l:3.3}
Если $p=\mathrm {char}(K)$ и $q=p^s\geq n-1, s\geq 1$, то в алгебре $F_2^{(n)}$ справедливы соотношения Фробениуса:
$$(a+b)^q=a^q+b^q,\qquad (ab)^q=a^qb^q,\qquad [a,b^q]=0.$$
\end{lemma}

Доказательство леммы см. в [16, §7 глава 5]. Впрочем она легко вытекает из двух предыдущих лемм.
Отметим также, что подробные доказательства этих трех лемм можно найти в [9].

\subsection{Теоремы о центрах.}
 \begin{theorem}\label{th:z1}
   Если $\mathrm {char}(K)=0$, то $Z(F_r^{(n)})=T^{(n-1)}(F_r^{(n)})$, где $r=2,3$.
 \end{theorem}

\begin{proof}
 Пусть, для определенности, $r=3$. Заметим, что в силу теоремы \ref{th:prod}: $T^{(n-1)}(F_3)\subseteq Z(F_3^{(n)})$.
  Допустим, что $0\ne f$~--- центральный многочлен алгебры $F_3^{(n)}$; тогда $[f,x]=0$. Применяя частные производные $\frac{\partial f}{\partial t}$, где $t\in X_3$, можно считать, что $f$ собственный многочлен. Тогда по теореме \ref{th:prod} для веса ${\rm wt}([f,x])$ выполнено неравенство ${\rm wt}([f,x])\ge n.$ Следовательно, ${\rm wt}(f)\ge n-1$ (при действии дифференцирования $D_x$ вес элемента увеличивается на единицу). Тогда вновь по теореме \ref{th:prod} справедливо включение $f\in T^{(n-1)}(F_3^{(n)})$.
\end{proof}

\begin{theorem}\label{th:z2}
  Пусть $K$~--- бесконечное поле конечной характеристики $p$, $s$~--- такое наименьшее число, что $q:=p^s\ge n-1$, $r=2,3$. Тогда
  \[Z(F_r^{(n)})=( T^{(n-1)}+Z_q)(F_r^{(n)}),\]
  где $Z_q$~--- $\rm T$-пространство, порожденное элементом $x^q$.
\end{theorem}

 \begin{proof}
  В силу леммы  \ref{l:3.3} и теоремы \ref{th:prod} имеем
  $$T^{(n-1)}+Z_q\subseteq Z(F_r^{(n)}).$$

Обратно, пусть $f\in Z(F_r^{(n)})$. Тогда по теореме Пуанкаре-Бирхгофа-Витта $f$  является линейной комбинацией правильных слов $v_i$: $f=\sum_i{\alpha_iv_i}$.

Элемент $f$ можно считать однородным и тогда среди слагаемых только одно может иметь вид $\alpha_0v_0$, где $\alpha_0\ne 0$, $v_0=x_1^{N_1}x_2^{N_2}x_3^{N_3}$, а все остальные $v_i$ содержат в своем разложении хотя бы один коммутатор. Если $N_1$ не делится на $q$, то полагая $x_2=x_3=1$, получаем центральный элемент $x_1^{N_1}$. Итак,
$$[x^{N_1},y]=0, \text { где } x=x_1,y=x_2.$$
Разделим $N_1$ на $q$ с остатком: $N_1=Mq+N, 0<N<q$. Поскольку $x^q\in Z(F_3^{(n)})$, то $x^{Mq}[x^N,y]=0$.
  Применяя лемму \ref{l:3.2}
  \begin{equation}\label{e5}
    0=x^{Mq}[x^N,y]=\sum_{i=1}^N{\binom {N}{i} x^{Mq+N-i}\underbrace{[x,y,x,\dots,x]}_{i+1}},
  \end{equation}
  получаем, что биномиальные коэффициенты $\binom {N}{i}$ для $i=1,\dots,N-1$ должны быть кратны числу $p$.

 Проверим, что $N=p^t$ и $N\ge n-1$. В самом деле, если $N=p^tm$, где $(p,m)=1$ и $t,m>1$, то $N=p^tm<q=p^s$ и $t<s$. Значит, $p^t<n-1$ и ввиду леммы \ref{l:3.1} и равенства \eqref{e5} получаем противоречие $[x^{N_1},y]\ne0$. Отсюда следует, что $N$ кратно $q$ и $\alpha_0v_0\in Z_q\subseteq Z(F_r^{(n)})$.

 Итак, можно считать, что для подходящих собственных многочленов $g_{(i_1,i_2,i_3)}\neq 0$ выполнено равенство
 $$f=\sum_{(i_1,i_2,i_3)}g_{(i_1,i_2,i_3)}x_1^{i_1}x_2^{i_2}x_3^{i_3}.$$

  Выбирая среди троек $(i_1,i_2,i_3)$ максимальную в смысле лексикографического порядка и применяя частные производные $\frac{\partial f}{\partial t}$, где $t\in X_3$, соответствующих порядков, получаем центральный собственный многочлен $g_{(i_1,i_2,i_3)}$. Тогда из теоремы \ref{th:prod} вытекает, что ${\rm wt}( g_{(i_1,i_2,i_3)})\ge n-1$, значит, $g_{(i_1,i_2,i_3)}\in T^{(n-1)}$ и $f\in T^{(n-1)}$.
 \end{proof}

\begin{corollary}
 Верно равенство: $Z(F_2^{(n)})=F_2^{(n)}\cap Z( F^{(n)})$.
\end{corollary}
\begin{proof}
 Пусть $f( x,y)\in Z(F_2^{(n)})$, тогда в силу доказанных теорем можно считать, что $f\in T^{(n-1)}(F_2^{(n)})$ . Применяя теорему \ref{th:prod}, имеем $[f,z]\in [T^{(n-1)}( F_3),F_3]\subseteq T^{(n)}( F_3)=0$, т.е. $f\in F_2^{(n)}\cap Z( F^{(n)})$. Следовательно, $Z(F_2^{(n)})\subseteq F_2^{(n)}\cap Z( F^{(n)})$. Обратное включение тривиально.
\end{proof}

\begin{remark}
  Для алгебры $F_3^{(3)}$ утверждение следствия неверно, поскольку $[x,y]z\in Z(F_3^{(3)})$ в силу теоремы \ref{th:prod}, но $[x,y]z\notin Z({F^{(3)}})$.
\end{remark}

\section{$\rm T$-пространства, связанные с алгеброй $F_2^{(3)}$}
\subsection{$\rm T$-пространства в алгебре многочленов $K[x_1,x_2]$.}

Пусть $U$~--- $\rm T$-пространство в алгебре многочленов $K[x_1,x_2]$ (свободная ассоциативная и коммутативная алгебра ранга 2). Элемент $u\in U$ является линейной комбинацией нормированных одночленов, т.е.
$$u=\sum_{(i_1,i_2)}\alpha_{(i_1,i_2)}x_1^{i_1}x_2^{i_2}.$$
Пусть $\alpha_{(i_1,i_2)}\neq 0$. Поскольку каждое $\rm T$-пространство, в частности, $U$ замкнуто относительно взятия однородных компонент, то $x_1^{i_1}x_2^{{i_2}}\in U$. Подставляя вместо одной из переменных единицу, получаем $x_1^{i_1}\in U,x_2^{{i_2}}\in U$. Допустим, что $i_1={p^s}m$, где $(p,m)=1$. Подставляя вместо $x_1$ сумму $x_1+1$ и вычисляя компоненту степени $p^s$ по переменной $x_1$, получаем
$\binom{p^sm}{p^s}x_1^{p^s}\in U$,
откуда в силу леммы \ref{l:3.1} верно $x_1^{p^s}\in U$. Считая, что $s$~--- наименьшее положительное число,
для которого $x_1^{p^s}\in U$, получаем $U=K[x_1^{p^s},x_2^{p^s}]$.

Таким образом, \textit{всякое собственное $\rm T$-пространство алгебры многочленов $K[x_1,x_2]$ порождается (как $\rm T$-пространство) элементом $x_1^{p^s}$, значит, совпадает с алгеброй $K[x_1^{p^s},x_2^{p^s}]$}.

Заметим, что указанный факт справедлив для алгебры многочленов от любого числа переменных (см. [3]).

\subsection{Предварительные леммы.}

Всюду далее, если не оговорено противное, $p=\text {char}(K)>0$. Всякий однородный многочлен из коммутаторного идеала алгебры $F_2^{(3)}$ пропорционален многочлену вида $f(q_1,q_2)=[x,y]{x^{q_1-1}}{y^{q_2-1}}$.
В [4] было доказано, что $\rm T$-пространство, порожденное многочленами
$f({p^s},{p^s}),\ s\in \mathbb{N}$,
бесконечно базируемо (т.е. не является конечно базируемым)
в алгебре $F_2^{(3)}$ (не содержащей единицу).

Многочлен $f(q_1,q_2)$ назовем \textit{особым}, если $q_1,q_2\in \{ {p^s}\mid s\ge 1\}$. Заметим, что можно считать $q_1 \leq q_2$, поскольку $f(q_1,q_2)$ и $f(q_2,q_1)$ очевидно $\rm T$-эквивалентны. В следующем пункте мы докажем, что всякое собственное $\rm T$-пространство алгебры $F_2^{(3)}$, содержащееся в коммутаторном идеале $T^{(2)}$, порождается особыми многочленами и коммутатором $[x,y]$.

Для этого нам потребуются две леммы, относящиеся к алгебре $F_2^{(3)}$.
\begin{lemma}\label{l:tc1}
а) Если $(q_1,p)=1$, то $f(q_1,q_2)$ является $\rm T$-следствием коммутатора $[x,y]$. В частности, неособый элемент является $\rm T$-следствием коммутатора.

  б) Если $(t_1,p)=1$, то $f(p^{s_1}t_1,q_2)$ является $\rm T$-следствием $f(p^{s_1},q_2)$.
\end{lemma}

\begin{proof}
Поскольку коммутатор $[x,y]$ централен в алгебре $F_2^{(3)}$, то применяя внутреннее дифференцирование $D_y$, получаем $$[x^{q_1},y]=q_1[x,y]{x^{q_1-1}},$$
откуда
$$q_1f(q_1,q_2)=q_1[x,y]x^{q_1-1}y^{q_2-1}=[x^q_1,y]{y^{q_2-1}}=[x^q_1y^{q_2-1},y].$$
Значит, если $(q_1,p)=1$, то $f(q_1,q_2)$ является $\rm T$-следствием коммутатора $[x,y]$, т.е. элемент $f(q_1,q_2)$ содержится в $\rm T$-пространстве, порожденном коммутатором $[x,y]$.

Аналогично, подставляя вместо элемента $x$ элемент $x^{t_1}$, получаем
\begin{gather*}
  f(p^{s_1},q_2)|_{x=x^{t_1}}=[x^{t_1},y]( x^{t_1})^{p^{s_1}-1}y^{q_2-1}=\\
  t_1[x,y]x^{t_1-1}(x^{t_1})^{p^{s_1-1}}y^{q_2-1}=t_1[x,y]x^{t_1-1}x^{p^{s_1}t_1-t_1}y^{q_2-1}=\\
  t_1[x,y]x^{p^{s_1}t_1-1}y^{q_2-1}=t_1f(p^{s_1}t_1,q_2).
\end{gather*}
Если $(t_1,p)=1$, то $f(p^{s_1}t_1,q_2)$ является $\rm T$-следствием элемента $f(p^{s_1},q_2)$.
\end{proof}

\begin{remark}
  Коммутатор можно получить как смешанные частные производные подходящих порядков из любого многочлен $f(q_1,q_2)$. Значит, из леммы \ref{l:tc1} следует, что $\rm T$-эквивалентны любой неособый многочлен вида $f(q_1,q_2)$ и коммутатор $[x,y]$, следовательно, $\rm T$-эквивалентны все неособые многочлены вида $f(q_1,q_2)$.
\end{remark}

\begin{lemma}\label{l:tccom}
Особый элемент не является $\rm T$-следствием коммутатора.
\end{lemma}

\begin{proof}
Пусть числа $q_1,q_2$ кратны $p$. Допустим, от противного, что особый элемент $f(q_1,q_2)$ является $\rm T$-следствием коммутатора. Тогда имеем представление
\begin{equation}\label{e6}
f(q_1,q_2)=\sum_i\lambda_i[x^{a_i}y^{b_i},x^{c_i}y^{d_i}]
\end{equation}
  для подходящих натуральных чисел $a_i,b_i,c_i,d_i$ и скаляров $\lambda_i\in K$.
 Поскольку коммутатор по каждой переменной является дифференцированием и элемент $[x,y]$ централен в алгебре $F_2^{(3)}$, то справедливо равенство
 \begin{equation}\label{e7}
[x^ay^b,x^cy^d]=\begin{vmatrix}
   a & b  \\
   c & d  \\
\end{vmatrix}[x,y]x^{a+c-1}y^{b+d-1}.
\end{equation}

  Теперь на основании равенств (\ref{e6}) и (\ref{e7}) можем записать
  \[f(q_1,q_2)=\sum_i\lambda_i\begin{vmatrix}
   a_i & b_i  \\
   c_i & d_i  \\
\end{vmatrix}[x,y]x^{a_i+c_i-1}y^{b_i+d_i-1}.\]
Сравнивая степени по переменным $x$ и $y$, получаем соотношения
$$a_i+c_i\equiv b_i+d_i\equiv0\pmod p.$$
Тогда
\[\begin{vmatrix}
   a_i & b_i  \\
   c_i & d_i  \\
\end{vmatrix}= \begin{vmatrix}
   a_i+c_i & b_i+d_i  \\
   c_i & d_i  \\
\end{vmatrix}\equiv0\pmod p .\]
Тем самым, получаем противоречие: $f(q_1,q_2)=0$.
\end{proof}

\begin{remark}
 Из леммы \ref{l:tccom} следует, что особый многочлен $f(q_1,q_2)$ не является $\rm T$-следствием никакого неособого многочлена вида $f(q_1,q_2)$.
\end{remark}

\subsection{$\rm T$-пространства, содержащиеся в коммутанте.}

\begin{theorem} \label{th:5}
Всякое ненулевое $\rm T$-пространство $U$ алгебры $F_2^{(3)}$, содержащееся в коммутаторном идеале $T^{(2)}$, порождается элементом $[x,y]$ и особыми многочленами из $U$.
\end{theorem}

\begin{proof}
Пусть $A(\Phi)=F_2^{(3)}(\Phi)$ и $A_{\omega}(\Phi)=F_{\omega}^{(3)}(\Phi)$~--- $\Phi$-свободные унитальные алгебры с тождеством $[x,y,z]=0$ над множествами
$$\{x,y\} \text { и } X=\{x,y,x_1,y_1,\dots,x_k,y_k,\dots\}$$
свободных порождающих соответственно. Поясним сначала, как получить
все $\rm T$-следствия однородного многочлена $f(x,y)$ в алгебре $A(\Phi)$. Каждое следствие многочлена $f(x,y)$ получается подстановкой в $f(x,y)$ вместо переменных $x,y$ некоторых значений
 $$x=\sum_{i=1}^k \alpha_iu_i,\quad y=\sum_{j=1}^l \beta_jv_j,$$
где $ \alpha_i,\beta_j\in \Phi$, $u_i, v_j$~--- нормированные одночлены алгебры $A$.
Указанную подстановку можно получить в виде композиции двух последовательных подстановок (гомоморфизмов соответствующих алгебр):
\begin{gather}
  \label{eq:1}
  x = x_1+\dots+x_k,\qquad y = y_1+\dots+y_l,\\
  \label{eq:2}
  x_i= \alpha_iu_i,\qquad y_j = \beta_jv_j.
\end{gather}

Заметим, что $f(\sum _{i=1}^k x_i,\sum _{j=1}^l y_j)$ является суммой однородных компонент
$$f^\sigma (x_1,\dots,x_k;y_1,\dots,y_l),$$
каждая из которых определяется набором $\sigma =(n_1,\dots,{n_k},{m_1},\dots,{m_l})$ натуральных чисел таких, что
 $$\deg_{x_i}f^{\sigma}={n_i}\quad (i=1,\dots,k),$$
 $$\deg_{y_j}f^{\sigma}=m_j\ \quad (j=1,\dots,l).$$
 Отметим также, что при этом должны быть выполнены равенства
  \[\sum_{i=1}^kn_i = n = \deg_x(f),\qquad\sum_{j=1}^l{{m_j}} = m = \deg_y(f).\]
  Проводя теперь вторую подстановку (\ref{eq:2}), в силу однородности многочленов $f^\sigma$ имеем
  \begin{gather*}
    f^\sigma (\alpha_1u_1,\dots,\alpha_ku_k;\beta_1v_1,\dots ,\beta_lv_l) = \\
    \prod (\alpha_i^{n_i})\prod (\beta_j^{m_j})f^\sigma (u_1,\dots,u_k;v_1,\dots,v_l).
  \end{gather*}
Итак, для того, чтобы описать все $\rm T$-следствия многочлена $f(x,y)$ в алгебре $A(\Phi)$ достаточно рассмотреть все его линеаризации $f^\sigma (x_1,\dots,x_k;y_1,\dots,y_l)$ и вычислить их значения в точках $x_i=u_i, y_j = v_j$. Модуль над $\Phi$, порожденный указанными значениями, совпадает с $\Phi$-модулем всех $\rm T$-следствий многочлена $f(x,y)$ в алгебре $A(\Phi)$.

Приступая теперь к доказательству теоремы, заметим, что в силу лемм \ref{l:tc1} и  \ref{l:tccom}
достаточно проверить, что особый многочлен из $U$ не является $\rm T$-следствием особых многочленов из $U$, в определенном смысле <<меньших>> данного.
На множестве особых многочленов вида $f(q_1,q_2),\, q_1\leq q_2$ введем отношение порядка $\prec$, считая
  $$f(q_1,q_2) \prec f(q_1',q_2')\Leftrightarrow q_1\leq q_1',\, q_2\leq q_2'.$$
  Допустим, что особый многочлен $f(r_1,r_2),\, r_1\leq r_2$ является  $\rm T$-следствием особого многочлена $f(q_1,q_2),\, q_1\leq q_2$.
 Значит, $f(r_1,r_2)$ получается из $f(q_1,q_2)$ в результате выполнения подстановок (\ref{eq:1}) и (\ref{eq:2}), некоторые из одночленов, входящих в равенства (\ref{eq:2}), могут быть равны $1$. Предположим, что мы берем однородную компоненту степени $(m_1,\dots,m_l)$ по переменным $y_1,\dots, y_l$. Если какая-то из этих переменных принимает значение $1$, то в качестве следствия $f(q_1,q_2)$ при вычислении частной производной некоторого порядка может возникнуть согласно теореме \ref{th:prod} особый многочлен, меньший данного в смысле отношения $\prec$. Значит, без ограничения общности, можно считать, что значениями переменных являются одночлены $x^a y^b \ne 1$.

  Тогда обозначая через $f(q_1,q_2)^{\sigma}$ однородную компоненту многочлена, полученного из $f(q_1,q_2)$ при этой замене, имеем
   $$\deg_{x_i}f(q_1,q_2)^{\sigma}={n_i},\quad \deg_{y_j}f(q_1,q_2)^{\sigma}=m_j\ (j=1,\dots,l).$$
  Кроме того, выполнены равенства
  \[\sum_{i=1}^kn_i=q_1,\qquad \sum_{j=1}^l{{m_j}}=q_2.\]
Положим $A=A(\mathbb{Z})$ и $A_{\omega}=A_{\omega}(\mathbb{Z})$.
Кроме того, обозначим через $\eta$ гомоморфизм колец $A_{\omega} \to A$,
определяемый заменами (\ref{eq:1}) и (\ref{eq:2}), где $u_i=x^{a_i}y^{b_i}$, $v_j=x^{c_j}y^{d_j}$. Тогда в кольце $A$ выполнено равенство:
 $$\eta (f(q_1,q_2)^{\sigma}) = L f(r_1,r_2),\: \text {где } \, L \in \mathbb{Z}.$$

 В кольце $A$ в силу (\ref{eq:2}) справедливы равенства:
  \[[u_i,v_j]=[x^{a_i}y^{b_i},x^{c_j}y^{d_j}]=\begin{vmatrix}
   a_i & b_i  \\
   c_j & d_j  \\
  \end{vmatrix}[x,y]x^{a_i+c_j-1}y^{b_i+d_j-1}.\]

  Опуская в записи множитель $x^{q_1-1}$, получим сумму элементов, возникающих из $[x,y]y^{q_2-1}$ подстановкой вместо $y$ элементов $y_1,\dots,y_l$ с указанными кратностями $(m_1,\dots,m_l)$.
Вычислим коэффициент $\nu_j$, с которым входит в сумму элемент
  $$[x,y_j]y_1^{m_1}\dots y_j^{m_j-1}\dots y_l^{m_l}.$$
Хорошо известно, что число возможных перестановок
$a_1 a_2 \ldots a_n$
из $n$ элементов, в которых содержится $n_1$ элементов первого рода,
$n_2$ элементов второго рода и т.д. $n_k$ элементов $k$-го рода, т.е. $n=\sum _{i=1}^k n_i$,
равно полиномиальному коэффициенту
$$\frac{n!}{n_1!\dots n_k!}.$$
Поэтому коэффициент $\nu_j$ равен
$$ \nu_j = \frac{(q_2-1)!}{(m_1)!\dots (m_j-1)! \dots (m_l)!}= \frac{m_j}{m_1 \dots m_l}N_2,$$
где $N_2=\frac{(q_2-1)!}{(m_1-1)!\dots(m_l-1)!}$.

  Пусть на первом месте в коммутатор поставлен элемент $x^ay^b$. Тогда коэффициент при базисном элементе $[x,y]x^N y^M$, который входит в сумму
  $$\sum_{j=1}^l{{\nu_j}[ {x^a}{y^b},{y_j}]y_1^{{m_1}}\dots y_j^{{m_j}-1}\dots y_l^{{m_l}}},$$
  равен
  \[\sum_{j=1}^l{ \begin{vmatrix}
   a & b  \\
   c_j & d_j  \\
  \end{vmatrix}m_j}\frac{N_2}{m_1\dots m_l}=\sum_{j=1}^l{ \begin{vmatrix}
   a & b  \\
   m_jc_j & m_jd_j  \\
  \end{vmatrix}\frac{N_2}{m_1\dots m_l}}= \begin{vmatrix}
   a & b  \\
   c & d  \\
  \end{vmatrix}\frac{N_2}{m_1\dots m_l},\]
  где $(c,d)=\sum_{j=1}^lm_j(c_j,d_j)$. Рассуждая аналогично, т.е. выполняя необходимую линеаризацию по $x$, и проводя указанную подстановку вместо переменных $x_1,\dots,{x_k}$ элементов $u_1,\dots,{u_k}$, получим
  \[L = \begin{vmatrix}
   a & b  \\
   c & d  \\
  \end{vmatrix}\frac{N_1}{n_1\dots n_k}\frac{N_2}{m_1\dots m_l},\]
  где $N_1=\frac{(q_1-1)!}{(n_1-1)!\dots({n_k}-1)!}$, $( a,b)=\sum_{i=1}^k{{n_i}({a_i},{b_i})}$. Следовательно,
  \[L = \begin{vmatrix}
   a+c & b+d  \\
   c & d  \\
  \end{vmatrix}\frac{N_1}{n_1\dots{n_k}}\frac{{N_2}}{m_1\dots m_l}= \begin{vmatrix}
   r_1 & r_2  \\
   c & d  \\
  \end{vmatrix}\frac{N_1}{n_1 \dots n_k}\frac{N_2}{m_1\dots m_l}=\]
  \[ \begin{vmatrix}
   1 & r_2r_1^{-1}  \\
   c & d  \\
  \end{vmatrix}\frac{r_1N_1}{n_1\dots n_k}\frac{N_2}{m_1\dots m_l}=\sum_{j=1}^l{ \begin{vmatrix}
   1 & r_2r_1^{-1}  \\
   m_jc_j & m_jd_j  \\
  \end{vmatrix}}\frac{r_1N_1}{n_1\dots n_k}\frac{N_2}{m_1\dots m_l}=\]
  \[\left(\sum_{j=1}^l{m_je_j}\right)\frac{r_1N_1}{n_1\dots n_k}\frac{N_2}{m_1\dots m_l}.\]
  Докажем, что ${L \in p\mathbb{Z}}$.
 Пусть $r_1={p^{{s_1}}}>q_1={p^{t_1}}$; тогда
$$\frac{r_1N_1}{n_1\dots n_k}={p^{s_1-t_1}}\frac{q_1!}{n_1!\dots {n_k}!}\in p\mathbb{Z},$$
$$\frac {m_jN_2}{m_1\dots m_l} = \frac{(q_2-1)!}{{m_1}!\dots(m_j-1)!\dots {m_l}!}\in \mathbb{Z}.$$
Отсюда немедленно получается требуемое $L \in p\mathbb{Z}$.

\end{proof}

\begin{corollary}
Коммутаторный идеал и центр алгебры $F_2^{(3)}$ являются бесконечно порожденными $\rm T$-пространствами.
\end{corollary}

Доказательство. Утверждение для коммутаторного идеала немедленно вытекает из теоремы \ref{th:5}.
  Центр $Z(F_2^{(3)})$ порождается как $\rm T$-пространство элементом $x^p$ и коммутантом ${T^{(2)}}(F_2^{(3)})$. Поскольку
  $$\{x^p\}^T = K[x^p, y^p],\qquad K[x^p, y^p]\cap {T^{(2)}}(F_2^{(3)})=0,$$
  то утверждение верно и для центра.

\begin{remark}
  В алгебре $F_2^{(3)}$ в случае $\mathrm {char}(K)=0$ коммутаторный идеал является единственным собственным ее  $\rm T$-подпространством.
\end{remark}

\section{Т-пространства в алгебре $F_2^{(n)}$}

\subsection{Лемма о дифференцированиях.}

\begin{lemma}\label{l:der}
 Пусть $\varphi(x_1,x_2,\dots,x_n)$~--- полилинейный собственный многочлен степени $n$ , $I_{n+1}$~--- идеал алгебры $F$, порожденный собственными многочленами степени $\ge n+1,n \in \mathbb{N}$. Тогда отображение $a\mapsto \varphi(a,x_2,\dots,x_n)$ является дифференцированием по переменной $x_1$ по модулю идеала ${I_{n +1}}$.
\end{lemma}
\begin{proof}
 Легко понять, что утверждение достаточно доказать для коммутатора $c=[x_1,\dots,x_{n+1}]$.
  Проведем индукцию по числу $n$. Если $n=1$, то утверждение тривиально, поскольку $D_b \colon a\mapsto [a,b]$ является дифференцированием.
  Пусть для коммутатора $c=[x_1,\dots,x_n]$ утверждение верно; тогда для $c(a)=c(a,x_2,\dots,x_n)$ можно записать
  $$\overline {c} = c(ab) -a\cdot c(b) - c(a)\cdot b\in I_n,\quad a ,b\in F,$$
  т.е. $\overline{c}$ является линейной комбинацией правильных элементов вида \eqref{e1}, в запись которых входят коммутаторы и не более одного раза встречаются элементы $a$ и $b$, которые считаются порождающими. Значит,
 $$\overline{c}=f\cdot a+g\cdot b+h,$$
 где $f,g,h$~--- собственные многочлены степени $\ge n$.
  Полагая $a=u\in[ F,F]$, получаем по модулю $I_{n+1}$:
$$c(ub)\equiv u\cdot c(b)+c(u)\cdot b,\qquad c(bu)\equiv b\cdot c(u)+c( b)\cdot u,$$
  значит,
  $$c(ub)\equiv c(u)\cdot b,\qquad c(bu)\equiv b\cdot c(u).$$
  Итак,
$$c(ub)\equiv c(u)\cdot b,\qquad c( bu)\equiv b\cdot c(u).$$
  Применяя указанные сравнения, получаем
$$c([xy,z])=c( x[y,z])+c([x,z]y)\equiv x\cdot c([y,z])+c([x,z])\cdot y,$$
но это и означает, что коммутатор степени $n+1$ обладает нужным свойством.
\end{proof}

\begin{remark}
  Идеал ${I_{n+1}}$ является $\rm T$-идеалом.
\end{remark}

\subsection{$\rm T$-нормальный ряд в алгебре $F_2^{(n)}$.}

В этом пункте обсуждается вопрос о существовании бесконечно базируемых $\rm T$-пространств, содержащихся в этажах $T^{(n)}/{T^{(n+1)}}$. Известно [4], что этаж $T^{(2)}/{T^{(3)}}$ содержит бесконечно базируемые $\rm T$-пространства, а этаж $T^{(3)}/{T^{(4)}}$ не содержит бесконечно базируемых $\rm T$-пространств при некоторых ограничениях на характеристику (см. [11]).

\begin{definition*}
Пусть $A$~--- относительно свободная алгебра и
  $U$~--- ее $\rm T$-пространство.
  Последовательность $\rm T$-пространств $0 = U_0\subset U_1\subset \dots\subset U_k\subset U_{k+1}=U$ назовем \textit{$\rm  T$-композиционным рядом пространства} $U$, если $U_{i+1}/U_i$ не содержат собственных $\rm T$-пространств.
\end{definition*}

\begin{definition*}
  $\rm T$-композиционный ряд $\rm T$-пространства $U$ в алгебре $A$ назовем \textit{$\rm T$-нормальным}, если все его члены являются идеалами в $A$.
\end{definition*}

 \begin{theorem} \label{th:norm }
   $\rm T$-идеал $T^{(3)}$ алгебры $F_2^{(n)}\ (n\ge 4)$ над полем $K$ характеристики $p\geq n$ обладает $\rm T$-нормальным рядом. В частности, всякое $\rm T$-пространство алгебры $F_2^{(n)}$, содержащееся в $T^{(3)}$, обладает конечной системой порождающих.
 \end{theorem}

\begin{proof}
 Пусть $A=F_2^{(n)}$. Упорядочим собственные многочлены $u_1,u_2,\dots$ вида \eqref{e1} алгебры $A$, содержащиеся в $T^{(n-1)}$, считая, что $u_i < u_j$, если $\deg u_i > \deg u_j$. В частности, наименьшим является элемент $u_1={{[xy]}^{n-2}}$ наибольшей степени. Заметим, что степень любого из многочленов $u_1,u_2,\dots$ по любой из переменных $x$ или $y$ не выше, чем $n-1$. Из конечной последовательности $u_1,u_2,\dots$ удалим те члены, которые являются $\rm T$-следствиями предыдущих. Все оставшиеся члены, за исключением, быть может, двух $u=[xy\dots y]$ и $v=[yx\dots x]$, имеют по каждой переменной степень $\ge 2$. Заменим $u$ на линеаризацию $u\Delta _y^{1}(x)$, которая имеет степень 2 по переменной $x$. Заметим, что ввиду ограничений на характеристику элементы $u=[xy\dots y]$ и $u\Delta _y^{1}(x)$ разумеется $\rm T$-эквивалентны. Тем самым, доказано, что все элементы $u_1,u_2,\dots$ по каждой из переменных $x,y$ имеют степени из интервала $[2,n-1]$.

  Пусть $T_k$~--- идеал алгебры $A$, порожденный элементами $u_1,\dots,u_k$. В силу замечания, сделанного после леммы \ref{l:der}, идеал $T_k$ является $\rm T$-идеалом. Пусть $u = u_{k+1}$. Поскольку $u\in Z(F_2^{(n)}/T_k)$ в силу леммы \ref{l:der} и порядка, введенного на множестве $\{u_1,u_2,\dots\}$, то имеем
$$T_{k+1} =AuA + T_k = uA + T_k = \sum_{(l,m)} \alpha_{lm}ux^ly^m + T_k.$$
   Отсюда следует, что если $U$~--- $\rm T$-пространство, содержащееся в факторе ${T_{k+1}}/T_k,$ то после применения подходящего числа частных производных получим $u\in U$.

   Линеаризуем элемент $u$ по переменной $x$ подстановкой $x = x_1 + x_2 + x_3$ и выберем однородную компоненту степени $(1,1,d-2)$, где $d=\deg_x(u)$. Вместо новых переменных подставим элементы $x_1 = x^i$, $x_2 = x^j$, $x_3 = x$ при некоторых фиксированных значениях $i$ и $j$. Следовательно, $d(d-1)ijux^{i+j-2}\in U$ в силу леммы \ref{l:der}. Тогда $iju{x^{i+j-2}}\in U$, поскольку $d < n \leq p$.
 Покажем, что для любого целого $s\ge 0$ разрешима система
  \[\{\begin{aligned}
  i+j-2=s,\, i,j\ge 1,\ i,j\notin p\mathbb{Z}
  \end{aligned}.\]
  В самом деле, если $s=pt+r$, где $r$~--- ненулевой остаток при делении $s$ на $p$, то решением указанной системы является пара $i=2,j=pt+r$. Если же $s=pt$, то решением системы является пара $i=1,j=pt+1$.

  Тогда для любого $s:u{x^s}\in U$. Рассуждая аналогично с переменной $y$, получаем, что для любых $s,t:u{x^s}{y^t}\in U$. Это доказывает, что фактор ${T_{k+1}}/T_k$ не имеет собственных $\rm T$-подпространств.

  Тем самым, построен нормальный ряд между $T^{(n-1)}$ и $T^{(n)}$. Аналогичным образом можно построить $\rm T$-нормальное уплотнение ряда $T^{(n-2)} \supset T^{(n-1)}$ и т.д.
\end{proof}

\begin{remark}
Покажем, что алгебра $A=F_2^{(n)}$ ранга 2 при $n\ge 5$ над бесконечным полем $K$ содержит бесконечно много $\rm T$-идеалов.

Во-первых, заметим, что элементы $[xy]^2,[xyxy]$ линейно независимы в $F_2^{(5)}$. Если $[xy]^2 = \gamma [xyxy]$, то $[x,y]^2\in Z({F_3^{(5)}})$, что неверно (см. [12], лемма 4). Пусть ${I_\alpha}$~--- $\rm T$-идеал в алгебре  $A=F_2^{(5)}$, порожденный многочленом
$f_\alpha = [xy]^2 + \alpha [xyxy]$, где $\alpha\in K$.
Покажем, что ${I_\alpha}={I_{\beta }}$ тогда и только тогда, когда $\alpha =\beta$. В самом деле, если ${I_\alpha}={I_{\beta }}$, то ${f_\alpha}=\lambda {f_{\beta }}$. Поскольку $[xy]^2,[xyxy]$ линейно независимы, то $\lambda =1$ и $f_\alpha = f_\beta$, т.е. $\alpha =\beta.$
\end{remark}

\textbf{Литература}

\begin{enumerate}

\item   В.Н. Латышев, О конечной порожденности $\rm T$-идеала с элементом $[x_1, x_2, x_3, x_4]$, \textit{Сиб. матем. журн.}, \textbf{6}:6 (1965), 1432--1434.

\item   И.Б. Воличенко, $\rm T$-идеал, порожденный элементом $[x_1, x_2, x_3, x_4]$, Препринт, Институт математики АН БССР. Минск. 1978.

\item   А.В. Гришин, Примеры не конечной базируемости $\rm T$-пространств и $\rm T$-идеалов в характеристике 2, \textit{Фунд. и прикл. матем.}, \textbf{5}:1 (1999), 101--118.

\item   В.В. Щиголев, Примеры бесконечно базируемых $\rm T$-пространств, \textit{Матем. Сб.}, \textbf{191}: 3 (2000), 143--160.

\item   В.В. Щиголев, Примеры бесконечно базируемых $\rm T$-идеалов, \textit{Фунд. и прикл. матем.}, \textbf{5}:1 (1999), 307--312.

\item   А.Я. Белов, О нешпехтовых многообразиях, \textit{Фунд. и прикл. матем.}, \textbf{5}:1 (1999), 47--66.

\item   А.В. Гришин, Бесконечно базируемое $\rm T$-пространство над полем характеристики 2 // Тезисы докладов международной конференции по алгебре и анализу, посвященной 100-летию со дня рождения Н.Г. Чеботарева (5-11 июня 1994 г., Казань). -- С. 29.

\item   А.В. Гришин, О строении центра относительно свободной алгебры Грассмана, \textit{Успехи мат. наук}, \textbf{65}:4 (2010), 191--192.

\item   А.В. Гришин, Л.М. Цыбуля, А.А. Шокола, О $\rm T$-пространствах и соотношениях Фробениуса в относительно свободных лиевски нильпотентных ассоциативных алгебрах, \textit{Фунд. и прикл. матем.}, \textbf{16}:3 (2010), 135--148.

\item   А.В. Гришин, О центре относительно свободной лиевски нильпотентной алгебры индекса 4, \textit{Матем. заметки}, \textbf{91}:1 (2012), 42--45.

\item   А.В. Гришин, О $\rm T$-пространствах в относительно свободной двупорожденной лиевски нильпотентной алгебре индекса 4, \textit{Фунд. и прикл. матем.}, \textbf{17}:4 (2012), 133--139.

\item   А.В. Гришин, С.В. Пчелинцев, О центрах относительно свободных ассоциативных алгебр с тождеством Ли нильпотентности, \textit{Матем. Сб.}, \textbf{206}:11 (2015), 113--130.

\item   А.В. Гришин, С.В. Пчелинцев, Собственные центральные и ядерные многочлены относительно свободных ассоциативных алгебр с тождеством Ли нильпотентности степени 5 и 6, \textit{Матем. сб}, \textbf{207}:12 (2016), 54--72.

\item  А.Р. Кемер, Конечная базируемость тождеств ассоциативных алгебр, \textit{Алгебра и логика}, \textbf{26}:5 (1987), 597--641.

\item   А.Р. Кемер, Тождества конечно порожденных алгебр над бесконечным полем, \textit{Изв. АН СССР, Сер. Матем.}, \textbf{54}:4 (1990), 726--753.

\item   Н. Джекобсон, \textit{Алгебры Ли}, М:, Мир, 1964.

\end{enumerate}

Пчелинцев Сергей Валентинович

Финансовый университет при Правительстве РФ, Москва,

Институт математики им. С.Л. Соболева, Сибирское отделение РАН, Новосибирск,

e-mail: pchelinzev@mail.ru

\end{document}